\documentclass[12pt,reqno]{amsart}
\usepackage[usenames]{color}

\usepackage{mathptmx}
\usepackage{bookman}

\usepackage[]{graphicx}
\usepackage[font=small,labelfont=bf]{caption}
\usepackage{sidecap}
\usepackage{float}
\usepackage{upgreek}

\input xy
\xyoption{all}

\usepackage{hyperref}

\usepackage{amssymb}
\usepackage{amscd}
\usepackage{array}
\usepackage{verbatim}
\theoremstyle{definition}
\newtheorem{thm}{Theorem}[section]
\newtheorem{lem}[thm]{Lemma}

\newtheorem{cor}[thm]{Corollary}
\newtheorem{rmk}[thm]{Remark}

\def\ccite#1{\textcolor{Red}{\cite{#1}}}

\numberwithin{equation}{section}
\begin{document}

\title[A\lowercase{nticommutator} N\lowercase{orm} 
\lowercase{of} P\lowercase{rojection} O\lowercase{perators}]
{\Large A\lowercase{nticommutator} N\lowercase{orm} F\lowercase{ormula}
\\ \lowercase{for} P\lowercase{rojection} O\lowercase{perators} 
}

\author{S\lowercase{am} W\lowercase{alters}}
\dedicatory{\SMALL {\rm University of Northern British Columbia}}
\date{April 3, 2016} 
\address{Department of Mathematics \& Statistics, University  of Northern B.C., Prince George, B.C. V2N 4Z9, Canada.}
\email[]{walters@unbc.ca}
%\dedicatory{}
\urladdr{http://hilbert.unbc.ca}

\begin{abstract} 
We prove that for any two projection operators $f,g$ on Hilbert space, their anticommutator norm is given by the formula \[\|fg + gf\| = \|fg\| + \|fg\|^2.\] The result demonstrates an interesting contrast between the commutator and anticommutator of two projection operators on Hilbert space. Specifically, the norm of the anticommutator $\|fg + gf\|$ is a simple quadratic function of the norm $\|fg\|$ while the commutator norm $\|fg - gf\|$ is not a function of $\|fg\|$. Nevertheless, the result gives the following bounds that are functions of  $\|fg\|$ on the commutator norm: $\|fg\| - \|fg\|^2 \le \|fg - gf\| \le \|fg\|$. 
\end{abstract}

\maketitle

%%%%%%%%%%%%%%%%%%%%%%%%%%%%%%%%%%%%%%%%%%%%%%%%%%%%%%%%%%%%%%%%%%%%
{\Large{\section{\bf The Main Result}}}

The main result of this paper is proving the following norm formula.

\medskip

\begin{thm}\label{fggf}
For any two projection operators $f,g$ on Hilbert space,
\begin{align}
\|fg + gf\| \ &= \ \|fg\| + \|fg\|^2.
\end{align}
\end{thm}

In particular, the anticommutator norm of projection operators is a simple quadratic function of the norm of their product. This is quite different from the commutator $fg - gf$ of projections since its norm is not a function of the norm of $fg$ (see remark below), nor is $\|fg\|$ a function of $\|fg-gf\|$. In view of Theorem \ref{fggf} we can nevertheless give bounds on the commutator norm that are functions of the norm $\|fg\|$. In this connection, the above theorem then has the following the consequence.

\begin{cor}\label{cor}
For any two projection operators $f,g$ on Hilbert space, one has
\[
\|fg\| - \|fg\|^2 \ \le \ \|fg - gf\| \ \le\ \|fg\| .
\]
\end{cor}

\begin{proof}
By Theorem 1.1 we have
\[
\|fg-gf\| = \|2gf - (fg+gf)\| \ge 2\|gf\| - \|fg+gf\| \ge \|fg\| - \|fg\|^2.
\]
Lemma \ref{lemma} below gives the inequality $\|fg - gf\| \le \|fg\|$.
\end{proof}

\medskip

(Lemma \ref{lemma} and its proof are given at the end of the paper.)

\medskip

We made an application of Theorem \ref{fggf} in \ccite{SWnearlyorthog} in order to obtain sharp upper bound estimates for projection operator on Hilbert space that are nearly orthogonal to their (unitary) symmetries. Specifically,  Theorem \ref{fggf} has lead us to obtain the relatively larger bound of $0.455$ that would guarantee that if a projection operator $e$ and a Hermitian unitary operator $w$ on Hilbert space satisfy $\|ewe\| < 0.455$, then a projection operator $q$ exists such that\footnote{The condition $qwq=0$, of course, just means that $q$ is orthogonal to its symmetric image $wqw^*$ under $w$. And so the condition that $\|ewe\|$ is ``small" means that $e$ is nearly orthogonal to its symmetry.} $qwq=0$ and $\|e-q\| \le \frac12\|ewe\| + 4\|ewe\|^2$. (Further, $q$ lies in the C*-subalgebra of $\mathcal B(\mathcal H)$ generated by $e$ and $wew^*$.)

\medskip

\begin{rmk}
It is not hard to see that the commutator norm of projections $\|fg - gf\|$ is generally not a function of the norm $\|fg\|$. For instance, if $f=g\not=0$, then $\|fg\|=1$ and their commutator is zero, while if $p,q$ are the two generating projections of the universal C*-algebra generated by two projections, then $\|pq\|=1$ and $\|pq-qp\| = \frac12$. Conversely, neither is the norm $\|fg \|$ a function of the norm $\|fg-gf\|$, since for the 2 by 2 matrix projections
$a = [\smallmatrix 1 & 0 \\ 0 & 0 \endsmallmatrix], \ 
b = \frac12[\smallmatrix 1 & 1 \\ 1 & 1 \endsmallmatrix]$, one has $\|ab-ba\| = \frac12 = \|pq-qp\|$ while $\|ab\| = \frac1{\sqrt2} \not= 1 = \|pq\|$. 
\end{rmk}

\begin{rmk}
We caution that it is not enough to check equalities such as that in Theorem \ref{fggf} for 2 by 2 matrices over the complex numbers and expect that they  generally hold for all projections. For example, one can check that the equation 
\begin{equation}\label{2by2}
\|fg - gf\|^2 = \|fg\|^2(1-\|fg\|^2)
\end{equation}
holds for all projections in $M_2(\mathbb C)$. One can, however, give simple examples of 4 by 4 projections for which equation \eqref{2by2} does not hold. Equation \eqref{2by2} also does not hold for the two projections $p,q$ of the universal C*-algebra generated by two projections since they do not commute and $\|pq\|=1$.
\end{rmk}

\noindent{\bf Acknowledgments.} 
This research was partly supported by a grant from the Natural Science and Engineering Council of Canada.

\bigskip

%%%%%%%%%%%%%%%%%%%%%%%%%%%%%%
%%%%%%%%%%%%%%%%%%%%%%%%%%%%%%
{\Large{\section{\bf Proof of $\|fg+gf\| \le \|fg\| + \|fg\|^2$}}}

\medskip

We shall use the following lemma.

\medskip

\begin{lem}\label{fg}
For any two projections $f,g$ and $m\ge1$ we have $\|(fg)^m\| \le \|fg\|^{2m-1}$.
\end{lem}
\begin{proof}
By induction, one checks the equality $(fg)^m = (fgf)^{m-1} (fg)$ (for $m\ge1$). Using $\|fgf\| = \|fg\|^2$, we have
\[
\|(fg)^m\| \le \|(fgf)^{m-1}\| \|fg\| \le \|fg\|^{2m-2} \|fg\| = \|fg\|^{2m-1}
\]
as required.
\end{proof}

We begin by observing and establishing the following formula for the powers of the anticommutator in terms of polynomials in the operators $fg, gf,$ $fgf$, and $gfg$:
\begin{equation}\label{powersplus}
(fg + gf)^n = P_n(fg) + P_n(gf) + Q_n(fgf) + Q_n(gfg)
\end{equation}
where $P_n, Q_n$ ($n=1,2,\dots$) are polynomials given recursively according to the dynamics
\begin{align}\label{polyPQ}
P_{n+1}(x) &= xP_n(x) + xQ_n(x), \\ 
Q_{n+1}(x) &= P_n(x) + xQ_n(x) \notag
\end{align}
with initial data $P_1(x) = x, \ Q_1(x) = 0$.\footnote{Interestingly, these polynomials turn out to be similar to Fibonacci polynomials as they will be given in very similar closed forms in terms of $\sqrt x$.} The equation \eqref{powersplus} can be checked by induction by making strong use of the fact that $f,g$ are projections. In order to find these polynomials in explicit form we express \eqref{polyPQ} in matrix form
\[
\bmatrix P_{n+1} \\ Q_{n+1} \endbmatrix =
\bmatrix x & x \\ 1 & x \endbmatrix
\bmatrix P_{n} \\ Q_{n} \endbmatrix.
\]
In order to telescope this expression we diagonalize the matrix here as follows: 
\[
\bmatrix x & x \\ 1 & x \endbmatrix = 
S \bmatrix x+\sqrt{x} & 0 \\ 0 & x-\sqrt{x} \endbmatrix S^{-1}, \qquad 
S := \bmatrix \sqrt{x} & -\sqrt{x} \\ 1 & 1 \endbmatrix
\]
(which is easily checked). 

Therefore, we calculate the polynomials as follows
\begin{align*}
\bmatrix P_{n+1} \\ Q_{n+1} \endbmatrix 
&=
\bmatrix x & x \\ 1 & x \endbmatrix^n
\bmatrix P_1 \\ Q_1 \endbmatrix
=
S \bmatrix (x+\sqrt{x})^n & 0 \\ 0 & (x-\sqrt{x})^n \endbmatrix S^{-1}
\bmatrix x \\ 0 \endbmatrix
\\ \\ 
&= 
\frac{1}{2\sqrt{x}}
\bmatrix \sqrt{x} & -\sqrt{x} \\ 1 & 1 \endbmatrix 
\bmatrix (x+\sqrt{x})^n & 0 \\ 0 & (x-\sqrt{x})^n \endbmatrix
\bmatrix 1 & \sqrt{x} \\ -1 & \sqrt{x} \endbmatrix
\bmatrix x \\ 0 \endbmatrix
\\ \\ 
&= 
\frac{1}{2\sqrt{x}}
\bmatrix \sqrt{x} & -\sqrt{x} \\ 1 & 1 \endbmatrix 
\bmatrix (x+\sqrt{x})^n & 0 \\ 0 & (x-\sqrt{x})^n \endbmatrix
\bmatrix x \\ -x \endbmatrix
\\ \\ 
&= 
\frac{1}{2\sqrt{x}}
\bmatrix \sqrt{x} & -\sqrt{x} \\ 1 & 1 \endbmatrix 
\bmatrix x(x+\sqrt{x})^n \\ -x(x-\sqrt{x})^n \endbmatrix
\end{align*}
yielding the closed forms
\begin{align*}
P_{n+1}(x) &= \frac{x}2 \Big[(x+\sqrt{x})^n + (x-\sqrt{x})^n\Big],
\\
Q_{n+1}(x) &= \frac{\sqrt{x}}{2} \Big[(x+\sqrt{x})^n - (x-\sqrt{x})^n\Big].
\end{align*}

Next, we express these using their binomial expansions:
\begin{align*}
(x+\sqrt{x})^n &= \sum_{j=0}^n {n\choose j} x^j x^{\frac12(n-j)}
\\
(x-\sqrt{x})^n &= \sum_{j=0}^n {n\choose j} x^j (-1)^{n-j} x^{\frac12(n-j)}
\end{align*}
Using the notation $\delta_2^k = \frac12 (1+(-1)^k)$ which is 1 when $k$ is even and 0 when $k$ is odd, we can write
\[
P_{n+1}(x) = \frac12 x \sum_{j=0}^n {n\choose j} x^j (1 + (-1)^{n-j}) x^{\frac12(n-j)} 
= 
x \sum_{j=0}^n {n\choose j} x^j \delta_2^{n-j} x^{\frac12(n-j)}.
\]
Let us choose odd $n = 2N-1$ so that
\begin{align*}
P_{2N}(x) &= 
x \sum_{j=0}^{2N-1} {2N-1\choose j} x^j \delta_2^{2N-1-j} x^{\frac12(2N-1-j)}.
\\
&= \sum_{j=0}^{2N-1} {2N-1\choose j} \delta_2^{j-1} x^{N+1 + \frac12(j-1)}
\intertext{Now put $j = 2\ell-1$ where $\ell=1,2,\dots, N$ to get}
P_{2N}(x)
&= \sum_{\ell=1}^{N} {2N-1\choose 2\ell-1}  x^{N+\ell}.
\end{align*}
Now
we can compute the norm estimate at $fg$ (or $gf$) as follows:
\[
\|P_{2N}(fg)\|
\le \sum_{\ell=1}^{N} {2N-1\choose 2\ell-1}  \|(fg)^{N+\ell}\|.
\]
Here we use the inequality $\|(fg)^m\| \le \|fg\|^{2m-1}$ from Lemma \ref{fg} to get
\[
\|P_{2N}(fg)\|
\le \sum_{\ell=1}^{N} {2N-1\choose 2\ell-1}  \|fg\|^{2N+2\ell - 1}
= \|fg\|^{2N-1} \sum_{\ell=1}^{N} {2N-1\choose 2\ell-1}  \|fg\|^{2\ell}
\]
we note that the same bound is the same number for $\|P_{2N}(gf)\|$. At this juncture we make use of the identity 
\[
A_N(a) := \sum_{\ell=1}^{N} {2N-1\choose 2\ell-1}  a^{2\ell} 
= \frac{a}{2}\big[ (1+a)^{2N-1} - (1-a)^{2N-1} \big] 
\]
which we shall call $A_N(a)$ for simplicity, where $a \ge 0$. We can then write 
\[
\|P_{2N}(fg)\| \le \|fg\|^{2N-1} A_N(\|fg\|)
\]
and we obtain
\[
\|P_{2N}(fg)\| + \|P_{2N}(gf)\| \ \le \ 2 \|fg\|^{2N-1} A_N(\|fg\|).
\]
Similarly we work out the norms $\|Q_{2N}(fgf)\|$ and $\|Q_{2N}(gfg)\|$.
\[
Q_{n+1}(x) = \frac{\sqrt{x}}{2} \sum_{j=0}^n {n\choose j} x^j (1- (-1)^{n-j}) x^{\frac12(n-j)}
= \sqrt{x} \sum_{j=0}^n {n\choose j} x^j \delta_2^{n-j-1} x^{\frac12(n-j)}
\]
and again inserting $n = 2N-1$:
\begin{align*}
Q_{2N}(x)
&= \sqrt{x} \sum_{j=0}^{2N-1} {2N-1\choose j} x^j \delta_2^{2N-1-j-1} x^{\frac12(2N-1-j)}
\\
&= \sum_{j=0}^{2N-1} {2N-1\choose j} \delta_2^{j} x^{N + \frac{j}2}
\intertext{put $j = 2\ell$ where $\ell = 0, 1, 2, \dots, N-1$:}
Q_{2N}(x)
&= \sum_{\ell=0}^{N-1} {2N-1\choose 2\ell}  x^{N + \ell}.
\end{align*}
The norm becomes (using $\|(fgf)^m\| \le \|fg\|^{2m}$)
\[
\|Q_{2N}(fgf)\|
\le \sum_{\ell=0}^{N-1} {2N-1\choose 2\ell}  \|(fgf)^{N + \ell}\|
\le \sum_{\ell=0}^{N-1} {2N-1\choose 2\ell}  \|fg\|^{2N + 2\ell}
\]
or 
\[
\|Q_{2N}(fgf)\|
\le \|fg\|^{2N} \sum_{\ell=0}^{N-1} {2N-1\choose 2\ell}  \|fg\|^{2\ell}
= \|fg\|^{2N} B_N(a)
\]
where we use the identity
\[
B_N(a) := \sum_{\ell=0}^{N-1} {2N-1\choose 2\ell}  a^{2\ell} 
= \frac{1}{2}\big[ (1+a)^{2N-1} + (1-a)^{2N-1} \big]
\]
which we shall call $B_N(a)$ for convenience.

Let's write $a = \|fg\|$. Then we get
\begin{align*}
\|(fg+gf)^{2N}\| &= \|P_n(fg) + P_n(gf) + Q_n(fgf) + Q_n(gfg)\|
\\
&\le 2 \|P_n(fg)\| + 2\|Q_n(fgf)\| 
\\
&\le 2 \|fg\|^{2N-1} A_N(\|fg\|) + 2\|fg\|^{2N} B_N(\|fg\|)
\\
&= 2 a^{2N-1} A_N(a) + 2 a^{2N} B_N(a)
\end{align*}
\begin{align*}
\ \ &= 2 a^{2N-1} \cdot \frac{a}{2}\big[ (1+a)^{2N-1} - (1-a)^{2N-1} \big]
+ 2 a^{2N} \cdot \frac{1}{2}\big[ (1+a)^{2N-1} + (1-a)^{2N-1} \big]
\\
&= a^{2N} \cdot \big[ (1+a)^{2N-1} - (1-a)^{2N-1} \big]
+ a^{2N} \cdot \big[ (1+a)^{2N-1} + (1-a)^{2N-1} \big]
\\
&= 2 a^{2N} \cdot (1+a)^{2N-1}.
\end{align*}
Taking $2N$-th roots,
\[
\|fg+gf\|  \le 2^{1/2N} a (1+a)^{1-\frac1{2N}} 
\]
which in the limit as $N \to\infty$ gives
\[
\|fg+gf\| \le \|fg\| + \|fg\|^2.
\]

\bigskip

%%%%%%%%%%%%%%%%%%%%%%%%%%%%%%
%%%%%%%%%%%%%%%%%%%%%%%%%%%%%%
{\Large{\section{\bf Proof of $\|fg+gf\| \ge \|fg\| + \|fg\|^2$}}}

Any two projection operators $f,g$ on Hilbert space $\mathcal H$ may be represented in matrix block forms as
\[
f = \bmatrix I & 0 \\ 0 & 0 \endbmatrix, \quad 
g = \bmatrix D & V \\ V^* & D' \endbmatrix
\]
with respect to the orthogonal decomposition $\mathcal H = \mathcal M \oplus \mathcal M^\perp$ where $\mathcal M$ is the range of $f$ (and $\mathcal M^\perp$ that of $1-f$). Here, $D, D'$ are positive operators on $\mathcal M$ and $\mathcal M^\perp$, respectively, and $V:\mathcal M^\perp \to \mathcal M$, satisfying the relations 
\[
D - D^2 = VV^*, \quad DV + VD' = V, \quad D' - D'^2 = V^*V.
\]
(in view of $g$ being a projection). Since
\[
fg(fg)^* 
= \bmatrix D & V \\ 0 & 0 \endbmatrix \bmatrix D & 0 \\ V^* & 0 \endbmatrix = \bmatrix D^2+VV^* & 0 \\ 0 & 0 \endbmatrix
= \bmatrix D & 0 \\ 0 & 0 \endbmatrix
\]
we see that $\|fg\|^2 = \|D\|$. We now work out the powers of the anticommutator as follows. First, we have
\[
fg + gf = \bmatrix D & V \\ 0 & 0 \endbmatrix + \bmatrix D & 0 \\ V^* & 0 \endbmatrix = \bmatrix 2D & V \\ V^* & 0 \endbmatrix.
\]
We observe that the powers of this anticommutator have the form
\[
(fg + gf)^n = \bmatrix F_n(D) & F_{n-1}(D)V \\ \star & \star \endbmatrix
\]
where $F_n(x)$ is a certain sequence of polynomials with integer coefficients -- with initial data $F_1(x) = 2x, F_0(x) = 1$. We do not need to know the $\star$ entries at the bottom of the matrix because we are interested in the northwestern\footnote{It is rather interesting in this case that the information regarding the anticommutator norm is contained in this block for large $n$.} block $F_n(D)$ of $(fg + gf)^n$. Multiplying
\[
\bmatrix F_n(D) & F_{n-1}(D)V \\ \star & \star \endbmatrix 
\bmatrix 2D & V \\ V^* & 0 \endbmatrix 
= \bmatrix 2DF_n(D) + F_{n-1}(D)VV^*  & F_n(D)V \\ \star & \star \endbmatrix
\]
or
\[
(fg + gf)^{n+1} = \bmatrix 2DF_n(D) + (D - D^2)F_{n-1}(D)  & F_n(D)V \\ \star & \star \endbmatrix
\]
from which we see that the polynomial sequence has the Fibonacci-type recursion relation
\[
F_{n+1}(x) =  2xF_n(x) + (x - x^2)F_{n-1}(x).
\]
Telescoping this as one would with ordinary Fibonacci numbers, one eventually gets
\[
F_n(x) 
= \frac12 x^{n/2} \left[ (\sqrt x + 1)^{n+1} - (\sqrt x - 1)^{n+1} \right].
\]
(Indeed, one can check this by induction.) This polynomial function is (for each $n$) an increasing function over the interval $[0,\infty)$. Therefore, since $D$ is a positive operator, we have
\[
\|(fg + gf)^n\| \ge \|F_n(D)\| = F_n(\|D\|) = F_n(\|fg\|^2) 
= \frac12 \|fg\|^n \left[ (\|fg\| + 1)^{n+1} - (\|fg\| - 1)^{n+1} \right]
\]
\[
= \frac12 \|fg\|^n (\|fg\| + 1)^{n+1} \left[ 1 - \left(\frac{\|fg\| - 1}{\|fg\| + 1}\right)^{n+1} \right].
\]
Taking $n$-th roots (noting that the anticommutator $fg+gf$ is a Hermitian operator)
\[
\|fg + gf\| \ge \frac1{2^{1/n}} \|fg\|\, (\|fg\| + 1)^{1+\frac1n} \left[ 1 - \left(\frac{\|fg\| - 1}{\|fg\| + 1}\right)^{n+1} \right]^{1/n}.
\]
Letting $n\to\infty$ the right side converges to $\|fg\| + \|fg\|^2$ (since $(1-c^n)^{1/n} \to 1$ for\footnote{If $-1<c<1$ then $-1< c^n \le |c| < 1$ for each $n\ge1$, which gives $0<1-|c| \le 1-c^n < 2$ and the result follows by taking $n$-th roots.} any $-1 < c < 1$).

This completes the proof of Theorem \ref{fggf}.

\bigskip

We end the paper with the proof of the lemma used by Corollary \ref{cor}.

\medskip

\begin{lem}\label{lemma}
For any two projections $f,g$ on Hilbert space, $\|fg - gf\| \le \|fg\|$. Further, $\|fg - gf\| = \|fg-fgf\|$.
\end{lem}
\begin{proof}
Write
\[
\|fg - gf\|^2 = \|(fg-gf)^*(fg-gf)\| = \| fgf + gfg - fgfg - gfgf \|
\]
and note that the operator in the last norm can be written as the sum of two orthogonal positive operators:
\[
fgf + gfg - fgfg - gfgf = fg(1-f)gf + (1-f)gfg(1-f) = uu^* + u^*u
\]
where $u = fg(1-f)$. So its norm is the max of the norms of each term, both of which are equal to $\|u\|^2$. Thus, 
\[
\|fg - gf\| = \|u\| = \|fg-fgf\| = \|fg(1-f)\| \le \|fg\|
\]
as needed.
\end{proof}

\medskip

\noindent{\bf Note added in proof.} A generous colleague pointed out to the author that with a bit more work, one can deduce the anticommutator norm formula from a theorem of Halmos in \ccite{Halmos}. Our proof, however, is self-contained and independent of this -- and the formula (simple as it is) seems to be unknown.

%%%%%%%%%%%%%%%%%%%%%%%%%%%%%%%%%%%%%%%%%%%%%%%%%%%%%%%%%%%%%%%%%%%%%%%%%%%%
%%%%%%%%%%%%%%%%%%%%%%%%%%%%%%%%%%%%    REFERENCES      %%%%%%%%%%%
%%%%%%%%%%%%%%%%%%%%%%%%%%%%%%%%%%%%    REFERENCES      %%%%%%%%%%%
%%%%%%%%%%%%%%%%%%%%%%%%%%%%%%%%%%%%    REFERENCES      %%%%%%%%%%%

\end{document}